\documentclass[12pt]{amsart}
\usepackage{amsmath,amssymb,amsthm, amscd, mathtools}
\usepackage{pifont,mathrsfs}
\usepackage{array,lastpage}
\usepackage{enumerate,xspace,pifont,shadow}
\usepackage{mathrsfs}
\usepackage{xfrac}
\usepackage{bm}
\usepackage [english]{babel}
\usepackage [autostyle, english = american]{csquotes}
\MakeOuterQuote{"}
\usepackage{xcolor}

\DeclareSymbolFont{AMSb}{U}{msb}{m}{n}
\DeclareMathSymbol{\Z}{\mathbin}{AMSb}{"5A}
\DeclareMathSymbol{\R}{\mathbin}{AMSb}{"52}
\DeclareMathSymbol{\N}{\mathbin}{AMSb}{"4E}
\DeclareMathSymbol{\Q}{\mathbin}{AMSb}{"51}

\newcommand{\NN}{\mathbb{N}}
\newcommand{\ZZ}{\mathbb{Z}}

\newcommand{\tp}{\textup{tp}}

\newcommand{\dcl}{\textup{dcl}}
\newcommand{\acl}{\textup{acl}}
\newcommand{\Th}{\textup{Th}}

\newcommand{\floor}[1]{\lfloor #1 \rfloor}

\newcommand{\mc}[1]{\mathcal{#1}}

\newcommand{\ob}[1]{\overline{#1}}

\def\Ind{\setbox0=\hbox{$x$}\kern\wd0\hbox to 0pt{\hss$\mid$\hss}
\lower.9\ht0\hbox to 0pt{\hss$\smile$\hss}\kern\wd0}

\def\Notind{\setbox0=\hbox{$x$}\kern\wd0\hbox to 0pt{\mathchardef
\nn=12854\hss$\nn$\kern1.4\wd0\hss}\hbox to
0pt{\hss$\mid$\hss}\lower.9\ht0 \hbox to
0pt{\hss$\smile$\hss}\kern\wd0}

\newtheorem{thm}{Theorem}[section]
\newtheorem{lem}[thm]{Lemma}
\newtheorem{cor}[thm]{Corollary}
\newtheorem{prop}[thm]{Proposition}

\newtheorem{fact}[thm]{Fact}

\newtheorem{quest}[thm]{Question}

\theoremstyle{definition}
\newtheorem{definition}[thm]{Definition}

\theoremstyle{remark}

\theoremstyle{remark}

\theoremstyle{remark}
\newtheorem{claim}[thm]{Claim}

\theoremstyle{remark}

\theoremstyle{remark}

\begin{document}
\bibliographystyle{plain}

\title[Topology of definalbe sets in OAGs of burden 2]{Topological properties of definable sets in ordered Abelian groups of burden 2}

\author{Alfred Dolich and John Goodrick}

\thanks{The first author's research was partially supported by PSC-CUNY Grant \#63392-00 51. The second author would like to thank the Universidad de los Andes for granting him paid leave (Semestre de Trabajo Académico Independiente) during which part of this research was carried out.}

\address{Dept. of Math and CS \\ Kingsborough Community College (CUNY) \\ 2001 Oriental Blvd.\\ Brooklyn, NY 11235}
\address{Department of Mathematics \\ CUNY Graduate Center \\ 365 5th Ave. \\ New York, NY 10016}
\email{alfredo.dolich@kbcc.cuny.edu}

\address{Departamento de Matemáticas \\ Universidad de los Andes \\ Carrera 1 No. 18A-12 \\ Bogotá, COLOMBIA 111711}
\email{jr.goodrick427@uniandes.edu.co}

\maketitle

\begin{abstract}

We obtain some new results on the topology of unary definable sets in expansions of densely ordered Abelian groups  of burden $2$. In the special case in which the structure has dp-rank $2$, we show that the existence of an infinite definable discrete set precludes the definability of a set which is dense and codense in an interval, or of a set which is topologically like the Cantor middle-third set (Theorem~\ref{sec2main}). If it has burden $2$ and both an infinite discrete set $D$ and a dense-codense set $X$ are definable, then translates of $X$ must witness the Independence Property (Theorem~\ref{IP_translations}). In the last section, an explicit example of an ordered Abelian group of burden $2$ is given in which both an infinite discrete set and a dense-codense set are definable.
\end{abstract}

\section{Introduction}

In this note we will study the topological properties of sets definable in densely ordered Abelian groups satisfying an extra model-theoretic hypothesis (having ``burden 2'') which in some sense limits the combinatorial complexity of combinations of instances of formulas (the precise definition will be recalled below). Typical examples of such groups are the structures $\mc{R}_1 = \langle \R; <, +, \Q \rangle$, the additive group of real numbers endowed with a unary predicate for the set $\Q$ of rationals, and $\mc{R}_2 = \langle \R; <, +, \Z\rangle$, the same group but with a predicate for the integers. In fact, both of the structures $\mc{R}_1$ and $\mc{R}_2$ are of \emph{dp-rank 2}, which is equivalent to having burden 2 and being NIP.\footnote{Note that it can be tricky to prove precise upper bounds on the burden or the dp-rank of a structure. See \cite{DG} for a detailed calculation of the dp-rank of the structures mentioned here.}

Recall that in an expansion of a divisible ordered Abelian group (or ``OAG'') of dp-rank 1, no infinite discrete subset of the domain can be definable  (see \cite{Goodrick_dpmin}), nor can any dense and codense subset be definable (by a result of Simon \cite{S}). However, the examples $\mc{R}_1$ and $\mc{R}_2$ above show that neither of these results hold for OAGs of dp-rank 2. One of the new results of this article is that in an expansion of a divisible OAG of dp-rank 2, there cannot be \emph{both} a definable infinite discrete set and a definable dense-codense set (see Theorem~\ref{sec2main} below).

The goal of this article is to understand the topological properties of unary definable sets in an expansion of a  densely ordered OAG $\mathcal{R} = \langle R; < , +, \ldots \rangle$ with burden at most $2$. Suppose that $X \subseteq R$ is definable in such a structure. The case when $X$ is open may be considered the ``nicest'' situation since the topological structure around any point in $X$ is as simple as possible. An expansion of a  densely ordered OAG in which every infinite definable $X \subseteq R$ has interior is called a \emph{viscerally ordered structure}. In our previous work  \cite{viscerality}, we undertook an extensive analysis of definable sets in a viscerally ordered structure, giving a cell decomposition theorem and showing that topological dimension has many desirable properties, justifying the intuition that this is the ``tamest'' possible case.

Suppose now that $\mathcal{R}$ is \emph{not} visceral and that furthermore $\Th(\mathcal{R})$ has finite burden. Let $X \subseteq R$ be definable with empty interior. Since infinite definable discrete sets in $\mathcal{R}$ cannot have accumulation points (see \cite[Corollary 2.13]{DG}), it follows easily that $X$ can be partitioned as $X=X_0 \cup X_1 \cup X_2$ with each $X_i$ definable so that:

\begin{enumerate}

\item $X_1$ is either empty or dense and codense in an open definable set $U$ with $X_1 \subset U$;

\item $X_2$ is either empty or discrete (and possibly finite); and 

\item $X_3$ is either empty or an infinite definable set which is nowhere dense and has no isolated points.

\end{enumerate}

In case (3), if $X_3$ is non-empty then the topological closure $\overline{X}_3$ is what we call a \emph{Cantor-like} definable set, namely it is a nonempty set which is closed, nowhere dense, and has no isolated points.

We conjecture that in an expansion of a  divisible OAG of dp-rank $2$, if there is an infinite definable set satisfying one of the three conditions above (being discrete, being dense-codense in an interval, or being Cantor-like), then there cannot be any other infinite definable set satisfying either of the other two (giving a basic trichotomy). Though we cannot quite prove this, in the next section we will show that the existence of an infinite definable discrete set precludes the definability of either a dense-codense or an infinite Cantor-like set (Theorem~\ref{sec2main}). In the case when $\mathcal{R}$ has burden $2$ and is definably complete, there can never be a definable Cantor-like set (Theorem~\ref{no_cantor_dc} below). Finally, in any OAG with burden $2$, if there is both an infinite discrete set $D$ definable in such a structure and also a definable $X$ which is dense and codense in some interval, then the Independence Property can be witnessed by translations of $X$ (Theorem~\ref{IP_translations}).

After proving the above general results in Section 2, the final section (Section 3) is devoted to the study of a concrete example of a divisible OAG of burden 2 in which both an infinite discrete set and a dense codense set are definable.

Note that the present work focuses on the topological properties of definable sets. Given an expansion of a OAG of burden $2$ in which an infinite discrete set $D$ is definable, it turns out that $D$ must have a very simple ``arithmetical'' structure, similar to the subsets of $\Z$ definable in Presburger arithmetic. For much more on this topic, see our previous article \cite{DG} or our recent preprint \cite{discrete_burden_2}, which can be seen as a companion to the present work.

\subsection{Notation and basic definitions}

Mostly we will follow standard notational conventions from model theory (as in \cite{Guide_NIP}, for example). ``Formulas'' and ``models'' are as in first-order logic and overlined variables ($\overline{x}, \overline{a}, \ldots$) denote finite tuples. Unlike some authors, a plain variable such as $x$ (not $\overline{x}$) is always a single variable, and we rarely work in $T^{eq}$. Also note that ``definable'' for us always means ``definable over some set of parameters.''

The abbreviation OAG stands for Ordered Abelian Group, which is a structure $\langle G; +, < \rangle$ consisting of an Abelian group $\langle G; +\rangle$ endowed with a total ordering $<$ which is translation invariant ($x < y$ implies that $x+z < y + z$). Some OAGs are \emph{discretely ordered} and have a least positive element (such as $\langle \Z; +, ,\rangle$), but in the present article all OAGs will be \emph{densely ordered}. Note that if an OAG is \emph{divisible} -- that is, for every $x \in G$ and every positive integer $n$, there is a $y \in G$ such that $ny = x$ -- then it is densely ordered, but there are examples of OAGs of finite dp-rank which are densely ordered and not divisible. See, for instance, \cite{Guide_NIP} or \cite{pure_OAGs} for more examples.

All topological properties of sets mentioned in this article (``open,'' ``dense,'' and so on) refer to the order topology generated by all open intervals.

At some points in Section 2 it will be useful to work in the Dedekind completion $\overline{R}$ of an OAG $\mathcal{R}$, for which the following notion from \cite{MMS} will be useful.

\begin{definition}
\label{ded_sort}
Suppose that $\mathcal{R} = \langle R; <, \ldots \rangle$ is a linearly ordered structure and $\{X_{\overline{a}} \, : \, \overline{a} \in Z\}$ is a definable family of subsets $X_{\overline{a}}$ of $R$, with $Z \subseteq R^n$ definable over $\emptyset$ and $X_{\overline{a}} = \{b \in R \, : \, \mathcal{R} \models \varphi(b; \overline{a}) \}$ for some formula $\varphi(x; \overline{y})$. Consider the relation $\sim$ on $n$-tuples from $Z$ such that $\overline{a}_1 \sim \overline{a}_2$ if and only if

$$ \forall y_1 \in  X_{\overline{a}_1} \exists y_2 \in X_{\overline{a}_2} \left( y_1 < y_2 \right) \wedge \forall y_2 \in  X_{\overline{a}_2} \exists y_1 \in X_{\overline{a}_1} \left( y_2 < y_1 \right).  $$

The relation $\sim$ is a definable equivalence relation on $Z$ which we informally think of as expressing that $\sup(X_{\overline{a}_1}) = \sup(X_{\overline{a}_2})$ where the suprema are calculated in the Dedekind completion $\overline{R}$ of $\mathcal{R}$, and defined so that $\sup(\emptyset) = - \infty$ and $\sup(X_{\overline{a}}) = + \infty$ when $X_{\overline{a}}$ is unbounded. Then $Z / \sim$, which is a sort in $\mathcal{R}^{eq}$, is a called a \emph{sort in $\overline{R}$.} We naturally identify such a sort with a subset of the Dedekind completion $\overline{R}$ of $\mathcal{R}$ with the induced ordering.

Thinking of sorts in $\overline{R}$ as sorts in $\mathcal{R}^{eq}$, we can also talk about definable subsets of sorts in $\overline{R}$, and of functions to and from such sorts.
\end{definition}

\begin{definition}
If $\mathcal{R} = \langle R; <, +, \ldots, \rangle$ is an expansion of an OAG, then $\mathcal{R}$ is \emph{definably complete} if for every nonempty definable subset $X \subseteq R$ which has an upper bound, $\sup(X) \in R$.
\end{definition}

The next observation (\cite[Proposition 2.2]{ivp}) will occasionally be useful.

\begin{fact}
\label{DC_divisible}
If $\mathcal{R}$ is an expansion of a densely ordered OAG which is definably complete, then $\mathcal{R}$ is divisible.
\end{fact}

The less commonly used definitions we need are those of \emph{burden} and \emph{dp-rank}. These notions originally go back to Shelah \cite{strong_dep}, but we will use the versions as given by Adler \cite{adler_strong_dep}. We recall them briefly here. Below, ``$T$'' always denotes some complete theory.

\begin{definition}
An \emph{ict-pattern of depth $\kappa$} is a sequence $\{ \varphi_i(\overline{x}; \overline{y}_i) \, : \, i < \kappa \}$ of formulas and a sequence $\{\overline{a}_{i,j} \, : \, i < \kappa, j < \omega \}$ of tuples from some model $\mathcal{M} \models T$ such that for every function $\eta \, : \, \kappa \rightarrow \omega$, the partial type

\begin{equation}
\{\varphi_i(\overline{x}; \overline{a}_{i, j})^{\textup{if } j = \eta(i)} \, : \, i < \kappa, j < \omega \}
\end{equation}
is consistent, where the exponent ``$\textup{if } j = \eta(i)$'' means that the formula is negated if $j \neq \eta(i)$. If $p(\overline{x})$ is a partial type, an ict-pattern as above is \emph{in $p(\overline{x})$} if every partial type as in (1) is consistent with $p(\overline{x})$.

The partial type $p(\overline{x})$ has \emph{dp-rank less than $\kappa$} if there is \textbf{no} ict-pattern of depth $\kappa$ in $p(\overline{x})$. If the least $\kappa$ such that the dp-rank of $p(\overline{x})$ is less than $\kappa$ is a successor cardinal, say $\kappa = \lambda^+$, then we say that the dp-rank of $p(\overline{x})$ is $\lambda$.

The dp-rank of the theory $T$ is the dp-rank of the partial type $x=x$ (in a single free variable $x$), and $T$ is \emph{dp-minimal} if its dp-rank is $1$.
\end{definition}

\begin{definition}
An \emph{inp-pattern of depth $\kappa$} is a sequence $\{ \varphi_i(\overline{x}; \overline{y}_i) \, : \, i < \kappa \}$ of formulas, a sequence $\{k_i \, : \, i < \kappa\}$ of positive integers, and a sequence $\{\overline{a}_{i,j} \, : \, i < \kappa, j < \omega \}$ of tuples from some model $\mathcal{M} \models T$ such that:

$\bullet$ For each $i < \kappa$, the ``$i$-th row'' $$\{\varphi_i(\overline{x}; \overline{a}_{i, j}) \, : \, j < \omega \}$$ is $k_i$-inconsistent; and

$\bullet$ For each function $\eta \, : \, \kappa \rightarrow \omega$, the partial type

\begin{equation}
\{ \varphi_i(\overline{x}; \overline{a}_{i, \eta(i)}) \, : \, i < \kappa\}
\end{equation}
is consistent.

If $p(\overline{x})$ is a partial type, an inp-pattern as above is \emph{in $p(\overline{x})$} if every partial type as in (2) is consistent with $p(\overline{x})$.

The partial type $p(\overline{x})$ has \emph{burden less than $\kappa$} if there is \textbf{no} inp-pattern of depth $\kappa$ in $p(\overline{x})$. If the least $\kappa$ such that the burden of $p(\overline{x})$ is less than $\kappa$ is a successor cardinal, say $\kappa = \lambda^+$, then we say that the burden of $p(\overline{x})$ is $\lambda$.

The burden of the theory $T$ is the burden of the partial type $x=x$ (in a single free variable $x$), and $T$ is \emph{inp-minimal} if its burden is $1$.

\end{definition}

\begin{fact}
(Adler, \cite{adler_strong_dep}) The dp-rank of a theory is less than some cardinal $\kappa$ if and only if $T$ is NIP. In case $T$ is NIP, the dp-rank of $T$ is equal to the burden of $T$. In particular, $T$ is dp-minimal if, and only if, $T$ is both NIP and inp-minimal.
\end{fact}

To mention some related work, fields of finite dp-rank have been recently classified by Johnson, who also showed that any valued field with finite dp-rank is Henselian \cite{Johnson_dp_finite}. Interesting examples of finite burden structures which are not finite dp-rank include pseudo real-closed fields \cite{samy_prcf}. For more background on these concepts and how they relate to NIP, see the introduction to our companion article \cite{discrete_burden_2} or the survey \cite{goodrick_survey}.

\section{Topologoical properties of definable sets in burden-2 OAGs}

In this section we will prove the general topological results for OAGs of burden 2 mentioned in the introduction.

Throughout this section, we will work under the following assumptions, unless otherwise stated:

\begin{itemize}

\item $\mathcal{R} = \langle R; <, +, \ldots\rangle$ is an expansion of a  densely-ordered OAG with complete theory $T$; 

\medskip

\item $\mathcal{R}$ is sufficiently saturated (generally just $|T|^+$-saturated will be enough);

\medskip

\item The burden of $\mathcal{R}$ is at most $2$.

\end{itemize}

Note that in this section, we do not generally assume that $\mathcal{R}$ is definably complete (unless we explicitly say so).


\begin{lem}
\label{rank_1_interval}
Suppose that $I \subseteq R$ is an interval such that $I$ has burden $1$. Then if $X \subseteq I$ is definable and nowhere dense, then $X$ is finite.
\end{lem}

\begin{proof}
This is essentially the same result as Lemma~3.3(1) of \cite{Goodrick_dpmin}, except that we allow $I$ to be any interval instead of the whole universe $R$.  The same proof as in \cite{Goodrick_dpmin} goes through, working within $I$.
\end{proof}

\begin{lem}
\label{rank_2_intervals}
If there is an infinite discrete set definable in $\mathcal{R}$ then there is $\epsilon>0$ so that $(0,\epsilon)$ has burden $1$.
\end{lem}

\begin{proof}
Suppose to the contrary that $D \subseteq R$ is infinite, discrete, and definable, and that for every $\epsilon > 0$ in $R$ there is an inp-pattern of depth $2$ consistent with $(0, \epsilon)$.

Now using $\omega$-saturation of $\mathcal{R}$ we can select an increasing sequence of elements $\{a_i \, : \, i \in \omega\} \subseteq D$ and an $\epsilon > 0$ such that for every $i$ we have $(a_i - 2 \epsilon, a_i + 2 \epsilon) \cap D = \{a_i\}$. We can now construct an inp-pattern of depth $3$ by attaching translated copies of the inp-pattern within $(0, \epsilon)$ onto each point of $a_i$ and adding a third row consisting of pairwise disjoint intervals, leading to a contradiction. 

More precisely, if $\varphi_0(x, \overline{b}_{0,j})$ and $\varphi_1(x, \overline{b}_{1,j})$ witness an inp-pattern of depth $2$ consistent with $(0, \epsilon)$, then for each $\ell \in \{0,1\}$ let $\varphi'_\ell(x, \overline{b}_{\ell,j})$ be the formula expressing ``there is a unique point $a \in D$ such that $a < x < a + \epsilon$, and for this unique point $a$, the formula $\varphi_\ell(x-a, \overline{b}_{\ell,j})$ holds.'' These will be the first two rows of our inp-pattern, and the third row will consist of the pairwise disjoint intervals $I_i = (a_i - \epsilon, a_i + \epsilon)$. The inconsistency of each row is easy to check, and if $c_{i,j} \in (0, \epsilon)$ satisfies the formula $\varphi_0(x, \overline{b}_{0,i}) \wedge \varphi_1(x, \overline{b}_{1,j})$, then for any $k \in \omega$ the point $a_k + c_{i,j}$ satisfies $\varphi'_0(x, \overline{b}_{0,i}) \wedge \varphi'_1(x, \overline{b}_{1,j})$ and lies within the interval $I_k$.
\end{proof}

This lemma has an immediate and useful corollary:

\begin{cor}  
\label{no_Cantor}
If there is $X \subseteq R$ definable, infinite, and discrete then there is no definable Cantor-like set in $\mathcal{R}$.
\end{cor}

\begin{proof}  
By the previous Lemma, there is $\epsilon > 0$ such that $(0,\epsilon)$ has burden $1$.  Suppose for contradiction that there is $Y$ definable and Cantor-like.  Translating $Y$ as necessary, we may assume that $Y \cap (0,\epsilon) \neq \emptyset$; and intersecting $Y$ with a closed subinterval of $(0, \epsilon)$, we may further assume that there is a Cantor-like definable subset of $(0, \epsilon)$. This contradicts Lemma~\ref{rank_1_interval}.
\end{proof}

Using the above Corollary, we can rule out the existence of any Cantor-like definable set in the case when $\mathcal{R}$ additionally satisfies definable completeness. First we prove a simple lemma. Recall that if $X$ is a subset of an ordered structure, a \emph{convex component} of $X$ is a maximal subset of $X$ which is convex.

\begin{lem}
If $X$ is a Cantor-like subset of $R$, then $R \setminus X$ has infinitely many convex components.
\end{lem}

\begin{proof}
Recall that Cantor-like sets are nonempty by definition, so we may pick some $a \in X$. Given that $X$ has no isolated points, it is either the case that (i) for every positive $\epsilon \in R$, the interval $(a-\epsilon, a)$ contains a point of $X$, or (ii)  for every positive $\epsilon \in R$, the interval $(a, a + \epsilon)$ contains a point of $X$. Without loss of generality we assume that (ii) occurs, and in case of (i) a similar proof will work.

Pick some $b_0 > a$. Since $X$ is nowhere dense, there is some element $c_0$ of $R \setminus X$ contained in the interval $(a, b_0)$. Given the assumption (ii) above, the convex component $C_0$ of $c_0$ in $R \setminus X$ cannot contain points arbitrarily close to $a$, an hence there is a point $b_1 \in R$ such that $a < b_1$ and $b_1$ is less than every element of $C_1$. Now repeat the argument above to find $c_1 \in R \setminus X$ contained in $(a, b_1)$, and continuing by induction we may find an infinite sequence of elements $c_1 > c_2 > \ldots$ of $R \setminus X$ all in distinct convex components.

\end{proof}

\begin{thm}
\label{no_cantor_dc} If $\mathcal{R}$ is a densely ordered, definably complete OAG with burden at most $2$, then there is no definable Cantor-like set in $\mathcal{R}$.
\end{thm}

\begin{proof}
Assume otherwise, and let $X \subseteq R$ be a definable Cantor-like set. By the previous lemma, $R \setminus X$ consists infinitely many convex components, each of which is open (as $X$ is closed). By definable completeness, each convex component of $R \setminus X$ is an interval; we call these the \emph{complementary intervals}. All but at most two complementary intervals are bounded, and if $I = (a,b)$ is a bounded complementary interval, then we can use the fact that $\mathcal{R}$ is divisible (Fact~\ref{DC_divisible} above) to define its midpoint $I_m = \frac{a+b}{2}$. The collection of all such midpoints $I_m$ is itself definable and comprises and infinite discrete set, yielding a contradiction to Corollary~\ref{no_Cantor}.
\end{proof}

Now we will focus on the case when $\mathcal{R}$ has dp-rank 2 (which, recall, is equivalent to NIP plus having burden $2$). In this case, we will show that if $R$ is divisible and defines an infinite discrete set, then it cannot also define a set which is dense and codense in some infinite interval.

To this end, we will use the \emph{Shelah expansion} $\mathcal{R}^{Sh}$ of the structure $\mathcal{R}$. This concept, and the facts below, are due to Shelah \cite{Dependent_Shelah}, but we will follow the notation and presentation of Simon \cite{Guide_NIP}.

\begin{definition}
\label{Sh_expansion}
Suppose that $\mathcal{R} \prec \mathcal{U}$ and $\mathcal{U}$ is $|R|^+$ -saturated. The \emph{Shelah expansion} $\mathcal{R}^{Sh}$ of $\mathcal{R}$ is the expansion of $\mathcal{R}$ with the following new predicates: for every partitioned formula $\varphi(\overline{x}; \overline{y})$ and every finite tuple $\overline{b} \in U^{|\overline{y}|}$, define a predicate $S_{\varphi(\overline{x}; \overline{b})} (\overline{x})$ on $R^{|\overline{x}|}$ such that $$\mathcal{R}^{Sh} \models S_{\varphi(\overline{x}; \overline{b})} (\overline{a}) \Leftrightarrow \mathcal{U} \models \varphi(\overline{a}; \overline{b}).$$

\end{definition}

The subsets of $R^n$ defined by the new basic predicates $ S_{\varphi(\overline{x}; \overline{b})}$ as above are called \emph{externally definable sets}. An important example for our purposes is that if $C \subseteq R$ is convex, then using $|R|^+$-saturation we may find $a, b \in U$ such that $(a,b) \cap R = C$, and thus $C$ is externally definable.

The next fact summarizes the important basic properties of Shelah expansions.

\begin{fact}
\label{Sh_expansion_properties}
(\cite{Dependent_Shelah}, and see also \cite{Guide_NIP}) Suppose that the complete theory of $\mathcal{R}$ is NIP.
\begin{enumerate}
\item The subsets of $R^n$ which are definable in the Shelah expansion $\mathcal{R}^{Sh}$ are independent of the choice of the saturated extension $\mathcal{U}$ in Definition~\ref{Sh_expansion}, and hence we may talk about ``the'' Shelah expansion.
\item The structure $\mathcal{R}^{Sh}$ admits elimination of quantifiers.
\item The structure $\mathcal{R}^{Sh}$ is NIP.
\item If $X \subseteq R^n$ is type-definable in $\mathcal{R}$, then the dp-rank of $X$ as calculated in $\mathcal{R}$ is equal to the dp-rank of $X$ as calculated in $\mathcal{R}^{Sh}$.
\end{enumerate}
\end{fact}

\begin{proof}
While complete proofs of (1), (2), and (3) can be found in \cite{Guide_NIP}, we take the opportunity to explain how (4) follows from (2) and (3). Note that a fact similar to (4) has been claimed by Onshuus and Usvyatsov (see \cite{onsh_usv}) but only in the special case when the theory is dp-minimal.

Let $d_1$ be the dp-rank of $X$ as calculated in the original structure $\mathcal{R}$, and let $d_2$ be the dp-rank of $X$ as calculated in the Shelah expansion $\mathcal{R}^{Sh}$ (each of which exists since their theories are NIP). On the one hand, if $\kappa$ is any cardinal and $d_1 \geq \kappa$, then there is an ict-pattern of depth $\kappa$ in $X$ with parameters from some elementary extension of $\mathcal{R}$, and the same array of formulas is an ict-pattern of depth $\kappa$ in the expanded language of $\mathcal{R}^{Sh}$. Therefore $d_1 \leq d_2$. 

On the other hand, suppose that there is an ict-pattern 

\begin{equation}
\{\varphi_i(\overline{x}; \overline{a}_{i,j}) \, : \, i < \kappa, \, j < \omega \}
\end{equation}
of depth $\kappa$ consistent with $X$ in some elementary extension $\mathcal{R}'$ of $\mathcal{R}^{Sh}$. By quantifier elimination, we may assume that
 $$\varphi_i(\overline{x}; \overline{y}) = S_{\psi_i(\overline{x}; \overline{y}; \overline{b}_i)}(\overline{x}; \overline{y})$$
  where $\overline{b}_i$ is a tuple of parameters from the $|R|^+$-saturated model $\mathcal{U} \succ \mathcal{R}$ used to define $\mathcal{R}^{Sh}$.

Working in $\mc{R}^{Sh}$, for any $m,n \in \NN$, any $m$-element subset $\{l_1, \ldots, l_m\}$ of $\kappa$, and any formula $\theta(\overline{x})$ in $X$, there are parameters $\ob{c}_{i,j} \in R$ with $1 \leq i \leq  m$ and $1 \leq j \leq n$ so that 
for any $\eta: \{1, \dots, m\} \to \{1, \dots, n\}$ there is  $\ob{d}_{\eta} \in \theta(R)$ so that 
$$\mc{R}^{Sh} \models  S_{\psi_{l_i}(\overline{x}; \overline{y}; \overline{b}_{l_i})}(\overline{d}_{\eta}; \overline{c}_{i,j}) \text{ if and only if } \eta(i)=j.$$
But then
\[\mc{U} \models \psi_{l_i}(\ob{d}_{\eta}, \ob{c}_{i,j}, \ob{b}_{l_i}) \text{ if and only if } \eta(i)=j.\]
Thus by compactness, in $\mc{U}$ the formulas $\psi_i(\ob{x}, \ob{y}, \ob{b}_i)$ form an ict-pattern with $\kappa$ rows consistent with $X$ and hence $d_1 \geq d_2$.

\end{proof}

The next fact is a slight generalization of a theorem proved by Simon \cite{S} which we will need for what follows.

The fact below was proved by Simon \cite{S} in the case when the entire structure is a dp-minimal divisible OAG, but the same proof can be relativized to convex definable subgroups $G$ to yield:

\begin{fact}
\label{Simon_interior}
If $G \subseteq R$ is a type-definable convex divisible subgroup of $R$, $\textup{dp-rk}(G) = 1$, and $X \subseteq G$ is infinite and definable, then $X$ has nonempty interior.
\end{fact}

\begin{proof}
Suppose that $G$ is as in the statement (type-definable, convex, divisible, and with $\textup{dp-rk}$ $1$).  By the comment just before Fact~\ref{Sh_expansion_properties}, the set $G$ is externally definable. By Fact~\ref{Sh_expansion_properties}~(4), working in $\mathcal{R}^{Sh}$, the dp-rank of $G$ is still $1$. Now consider the structure $\mathcal{G}$ whose universe is the set $G$ and with the induced definable structure: that is, the basic predicates in the language for $\mathcal{G}$ represent sets of the form $G^n \cap Z$ where $Z \subseteq R^n$ is $\emptyset$-definable in the language of $\mathcal{R}$. The fact that $\textup{dp-rk}(G) = 1$ as calculated in $\mathcal{R}^{Sh}$ implies that the theory of the structure $\mathcal{G}$ is dp-minimal. Also the set $X \subseteq G$ is definable in $\mathcal{G}$. Now we can apply Theorem~3.6 of \cite{S} to conclude that $X$ has nonempty interior, as we wanted.
\end{proof}

Now we state our main result on divisible OAGs of dp-rank $2$.

\begin{thm}\label{sec2main}  If $\mathcal{R}$ is an expansion of a divisible ordered Abelian group of dp-rank at most $2$ and there is an infinite definable discrete set in $\mathcal{R}$ then there is no $X$ definable which is dense and codense in an interval $I$ of $\mathcal{R}$ nor is there a definable Cantor-like set.
\end{thm}

\begin{proof}
Say $X \subseteq R$ is and $\mathcal{R}$-definable. By Corollary~\ref{no_Cantor}, we know that $X$ is not Cantor-like. It only remains to consider the case when $X$ is dense in some interval $I$.

First pick $\epsilon \in R$ as in the conclusion of Lemma~\ref{rank_2_intervals} so that $(-\epsilon, \epsilon)$ has dp-rank $1$. Now let $G$ be the subset of $R$ defined as

$$G = \bigcap_{n \in \N} \left(-\frac{\epsilon}{n+1}, \frac{\epsilon}{n+1} \right)$$
and note that $G$ is a convex subgroup of $R$. By $\omega$-saturation, $G$ is infinite. Also, $G$ is definable in $\mathcal{R}^{Sh}$, so by Fact~\ref{Sh_expansion_properties}, it has dp-rank $1$ as calculated in the Shelah expansion.

Suppose that $X$ is dense and codense in $I$. Translating and truncating $I$ as necessary, we may assume that $I \subseteq G$. But then $X \cap I$ is an infinite definable subset of $G$ which has empty interior, contradicting Fact~\ref{Simon_interior}.

\end{proof}

We have considerably more precise results if we also assume that $\mathcal{R}$ is definably complete.

\begin{cor}\label{discrete_interior_dichotomy}
Suppose that $\mathcal{R}$ is an expansion of a definably complete divisible OAG of dp-rank at most $2$ in which an infinite discrete set is definable. Then for any definable $X \subseteq R$, either $X$ is discrete or $X$ has nonempty interior.
\end{cor}

\begin{proof}
Suppose that $X \subseteq R$ is definable and has empty interior. By Theorem~\ref{sec2main}, $X$ is nowhere dense. By Corollary~2.13 of \cite{DG}, it follows that $X$ is discrete.
\end{proof}

\begin{cor}\label{opencore}
Suppose that $\mathcal{R}$ is a definably complete expansion of a divisible ordered Abelian group with dp-rank at most $2$.  If there is $X \subseteq R$ definable which is dense and codense in some interval then any model of $T=\Th{(\mc{R})}$ has o-minimal open core.
\end{cor}

\begin{proof}
We claim that the theory of $\mathcal{R}$ has uniform finiteness: otherwise, there would exist a formula $\varphi(x; \overline{y})$ such that for every $n \in \N$, there are parameters $\overline{b}_n$ such that the set defined by $\varphi(x; \overline{b}_n)$ is finite and of size at least $n$. Note that there is a first-order formula $\theta(\overline{y})$ which expresses the property ``the set defined by $\varphi(x; \overline{y})$ is topologically discrete,'' and since any finite set is discrete, $\theta(\overline{b}_n)$ holds for each $n$. Thus every finite subset of

$$\{\exists^{\geq m} x \, \,  \varphi(x; \overline{y}) \, : \, m \in \N \} \cup \{\theta(\overline{y}) \}$$
is satisfiable by some tuple $\overline{b}_n$, and so by compactness and $\omega$-saturation there is a tuple $\overline{b}$ such that $\varphi(x; \overline{b})$ defines an infinite discrete set, contradicting Theorem~\ref{sec2main}.

Therefore by~\cite[Theorem A]{DMS}, $\mathcal{R}$ has o-minimal open core.
\end{proof}

We note that Theorem \ref{sec2main} no longer holds once the dp-rank of $\mathcal{R}$ exceeds $2$: for example, in \cite{DG} we have shown that the structure $\langle \R; +, <, \Q, \Z\rangle$ has dp-rank $3$.

For the remainder of this section, we return to the general situation when $\mc{R}$ is only assumed to have burden $2$ (rather than dp-rank $2$). Once again adapting arguments of Simon from \cite{S} allows us to obtain information on dense-codense definable sets in the case when there is also an infinite definable discrete set.

We begin with a definition.

\begin{definition}
\label{bounded_away_zero}
Let $X \subseteq R$ be definable and $f:X \to \ob{R}$ be a definable function from $X$ to some sort in the Dedekind completion $\ob{R}$ of $\mathcal{R}$ (see Definition~\ref{ded_sort} above) with $f(x)>0$ for all $x \in X$.   We say that \emph{$f$ is  bounded away from $0$} if whenever $I$ is an interval so that $I \cap X$ is infinite there is $\epsilon>0$ and a subinterval $J \subseteq I$ with $J \cap X$ infinite so that $f(a)>\epsilon$ for all $a \in J$.  

We say "functions on $X$ are bounded away from $0$" to mean that all definable $f: X \to \ob{R}$ with $f$ positive on $X$ are bounded away from $0$.
\end{definition}

\begin{lem}
\label{bdd_away_from_zero_bd_1}
If $X \subseteq R$ has burden $1$, then definable functions on $X$ are bounded away from $0$.
\end{lem}

\begin{proof}
We will adapt the proof of \cite[Lemma 3.19]{Goodrick_dpmin}. Suppose that definable functions on $X$ are not bounded away from $0$, as witnessed by a definable function $f \, : \, X \rightarrow \overline{R}$ and an interval $I$ such that:

(*) $I \cap X$ is infinite, and for any positive $\epsilon$ and every subinterval $J$ with $X \cap J$ infinite, there is some $a \in X \cap J$ such that $f(a) \leq \epsilon$.

By saturation, we may pick a sequence $\langle J_i \, : \, i \in \omega \rangle$ of pairwise disjoint subintervals of $I$ such that for each $i$ the set $J_i \cap X$ is infinite. The property (*) transfers to subintervals of $I$ so (*) is true, \emph{mutatis mutandis}, of each $J_i$ as well.

Pick $\epsilon_0 > 0$ in $R$ arbitrarily. By (*), we can pick an element $a_{i,0} \in J_i$ for each $i$ such that for each $i$,

\begin{equation}
f(a_{i,0}) < \epsilon_0.
\end{equation}

Now we pick positive elements $\epsilon_j$ and $a_{i,j} \in J_i$ for each $j \in \omega$ by induction, as follows: suppose we have already selected elements

$$\epsilon_0 > \epsilon_1 > \ldots > \epsilon_j$$
and elements $a_{i,j} \in J_i$ such that $f(a_{i,j}) < \epsilon_j$. By saturation, we can pick $\epsilon_{j+1} \in R$ such that for every $i \in \omega$,

\begin{equation}
0 < \epsilon_{j+1} < f(a_{i,j}). 
\end{equation}

Finally, applying the property (*) again, for each $i \in \omega$ we can pick an element $a_{i,j+1} \in J_i \cap X$ such that $f(a_{i, j+1}) < \epsilon_{j+1}$.

Now we have elements $\langle a_{i,j} \, : \, (i,j) \in \omega \times \omega \rangle$ such that for every $(i,j) \in \omega \times \omega$,

$$\epsilon_{j+1} < f(a_{i,j}) < \epsilon_j.$$

From this, we can construct an inp-pattern of depth $2$, as follows: in the first row, we use formulas $\varphi_0(x; \overline{b}_i)$ expressing the fact that $x \in J_i \cap X$, and $\{\varphi_0(x; \overline{b}_i) \, : \, i \in \omega \}$ is $2$-inconsistent since the $J_i$ are pairwise disjoint. In the second row, we use formulas $\varphi_1(x; \epsilon_j, \epsilon_{j+1})$ expressing the property that

$$\epsilon_{j+1} < f(x) < \epsilon_j$$
and again $\{\varphi_1(x; \epsilon_j, \epsilon_{j+1}) \, : \, j \in \omega \}$ is $2$-inconsistent. Furthermore, for any $(i,j) \in \omega \times \omega$, the formula $$\varphi_0(x; \overline{b}_i) \wedge \varphi_1(x; \epsilon_j, \epsilon_{j+1})$$ is consistent since it is satisfied by $a_{i,j}$. Thus we have an inp-pattern of depth $2$ in $X$, contradicting the assumption that $X$ has burden $1$.

\end{proof}

The next lemma generalizes a key fact  from \cite{Goodrick_dpmin} about definable functions in the dp-minimal case.

\begin{lem}
\label{bdd_away_from_zero_discrete}
If there is an infinite discrete set definable in $\mathcal{R}$, then functions on $R$ are bounded away from $0$.

\end{lem}

\begin{proof}
Suppose, to the contrary, that $f : R \rightarrow \overline{R}$ is definable, $f(x) > 0$ for every $x \in R$, and $f$ is not bounded away from $0$. Then there is a nonempty interval $I$ such that for any $\epsilon > 0$ and any subinterval $J$ of $I$, there is $a \in J$ such that $f(a) \leq \epsilon$. Therefore, $f$ is not bounded away from $0$ on any subinterval of $I$. By Lemma~\ref{bdd_away_from_zero_bd_1}, any subinterval of $I$ has burden $2$. By  Lemma~\ref{rank_2_intervals}, there is some interval $(0, \epsilon)$ of burden $1$, and shrinking $\epsilon$ as necessary, we may assume that $\epsilon$ is less than the diameter of $I$. Since burden is translation-invariant, we conclude that there is a subinterval of $I$ of burden $1$, a contradiction.

\end{proof}

We record an immediate consequence of Lemma \ref{bdd_away_from_zero_discrete}.

\begin{cor}\label{away_dense} Suppose there is an infinite discrete set definable in $\mathcal{R}$.  Suppose that $I$ is an interval and $X \subseteq I$ is definable, dense, and codense in $I$.  Then functions on $X$ are bounded away from $0$.
\end{cor}

\begin{proof}
If $g \, : \, X \rightarrow \overline{R}$ is definable and has positive values, pick any $a > 0$ and extend $g$ to a definable function $f \, : \, R \rightarrow \overline{R}$ by the rule that $f(x) = a$ when $x \notin X$. Now apply the previous Lemma to conclude that $f$, and hence $g$, is bounded away from $0$.
\end{proof}

\begin{definition}  For a definable set $X$ we write $a \sim_{X, \delta} b$ or simply $a \sim_\delta b$ for the equivalence relation on $R$ defined by $$ \forall \, \epsilon \in (-\delta, \delta) \left[ a + \epsilon \in X \Leftrightarrow b + \epsilon \in X \right],$$ and $a \sim_X b$ (or simply $a \sim b$) means that for some $\delta > 0$, we have $a \sim_\delta b$. 
\end{definition}

First we note the following easy but useful fact:

\begin{lem}
\label{translation_fact}
If $X$ is definable, $a, b \in R$, $a \sim_{X,\delta} b$, and $|\epsilon| < \delta$, then  $a + \epsilon \sim_{X,\delta - |\epsilon|} b + \epsilon$. In particular, $a + \epsilon \sim_X b + \epsilon$.
\end{lem}

\begin{lem}
\label{single_class_reduction}
Let $X$ be definable.  Fix any $a \in X$ and let $\widetilde{X} = \{b \in X \, : \, b \sim_X a\}$. Then $|\widetilde{X} / \sim_{\widetilde{X}} | = 1$.
\end{lem}

\begin{proof}
Let $b \in \widetilde{X}$, and we will show that $b \sim_{\widetilde{X}} a$. Choose $\delta > 0$ such that $b \sim_{X, \delta} a$ and suppose that $\epsilon \in (-\delta, \delta)$. By Lemma~\ref{translation_fact}, we have $ b + \epsilon \sim_X a + \epsilon$, so in particular $b + \epsilon \sim_X a$ if and only if $a + \epsilon \sim_X a$. In other words, $b + \epsilon \in \widetilde{X}$ if and only if $a + \epsilon \in \widetilde{X}$, so $b \sim_{\widetilde{X},\delta} a$.
\end{proof}

Now we look for type-definable subgroups related to definable subsets of $R$. The proof of the next lemma is similar to that of Theorem~3.6 from \cite{S}.

\begin{lem}
\label{single_class_group}
Suppose that:
\begin{enumerate}
\item $X \subseteq R$ is definable;
\item $I$ is an open interval such that $X \cap I$ is infinite; 
\item functions on $X \cap I$ are bounded away from $0$;
\item $X \cap I$ is not discrete; and
\item $|(X \cap I) / \sim| = 1$.
\end{enumerate}

Then for any $a \in X \cap I$, there is a nonzero type-definable convex subgroup $C$ of $(R, +)$ such that $(X - a) \cap C$ is also a subgroup of $(R, +)$.

\end{lem}

\begin{proof}
Fix $a \in X \cap I$ and define $f: X \rightarrow \overline{R}$ to be the function $$f(x) = \sup \{ \delta \in R \, : \, x \sim_\delta a \}.$$ 
Note that $f(x)>0$ on $I$.  Since functions on $X \cap I$ are bounded away from zero, we may pick an $\epsilon > 0$ and an open subinterval  $J \subseteq I$ such that $X \cap J$ is infinite and $f(x) > \epsilon$ for every $x \in J \cap X$.

\begin{claim}
There is a positive element $\epsilon_0 < \epsilon$ such that $f(x) > \epsilon_0$ for every $x \in (a - \epsilon_0, a+ \epsilon_0)$.
\end{claim}

\begin{proof}
 First, we may assume that $f(a) > \epsilon$ (by replacing $\epsilon$ by a positive element less than $f(a)$ if necessary). Pick some $b \in X \cap J$ and some $\epsilon' > 0$ such that $a \sim_{\epsilon'} b$, $(b - \epsilon', b + \epsilon') \subseteq J$, and $\epsilon' \leq \epsilon$. Now pick a positive element $\epsilon''$ such that $2 \epsilon'' < \epsilon'$ and let $z$ be an arbitrary element of $(-\epsilon'', \epsilon'')$. Then we have
 
 \begin{equation}
 a + z \sim_{\epsilon''} b + z
 \end{equation}
 since $a \sim_{\epsilon'} b$ and using Lemma~\ref{translation_fact}. On the other hand,  since $b + z \in J$, we also have that $f(b+z) > \epsilon > \epsilon''$ (by our initial choice of the interval $J$), and so
  
  \begin{equation}
 b + z \sim_{\epsilon'} a.
 \end{equation}

By (6), (7), and transitivity, we conclude that $a+z \sim_{\epsilon''} a$. Therefore, given that $f(a) > \epsilon > \epsilon''$, it follows that for any $a + z \in (a - \epsilon'', a + \epsilon'')$, we have $f(a+z) \geq \epsilon''$. Then if $0 < \epsilon_0 < \epsilon''$ and $a + z \in (a - \epsilon_0, a + \epsilon_0)$, we have $f(a + z) > \epsilon_0$, as desired.
 \end{proof}

Pick $\epsilon_0$ as in the Claim above and small enough so that $(a - \epsilon_0, a + \epsilon_0) \subseteq I$. Replace $I$ by the subinterval $(a - \epsilon_0, a + \epsilon_0)$ and note that all the hypotheses of the Lemma still hold (since (4) and (5) imply that $X \cap I$ has no isolated points, hence $X \cap (a - \epsilon_0, a + \epsilon_0)$ is infinite, and the rest are trivial).

Define
 $$C = \bigcap_{n \in \N}\{x \in R : -\epsilon_0 < nx < \epsilon_0\}$$
 and note that $C$ is closed under addition by the triangle quality. Thus $C$ is a convex type-definable subgroup of $(R, +)$. Let $H = (X - a) \cap C$ and suppose that $g, h \in H$. Then $g + h \in C$. Furthermore, since $f(a+g) > \epsilon_0$ (by the conclusion of the Claim above), we have $a + g \sim_{\epsilon_0} a$. Hence by Lemma~\ref{translation_fact} and the fact that $|h| < \epsilon_0$, we obtain $$a + g + h \sim a + h .$$ But $a+h \in X$ since $h \in X - a$, so $g + h \in X - a$. This shows that $H$ is closed under addition. For closure under negation, if $g \in (X-a) \cap C$, then it is immediate that $-g \in C$, while $a + g \sim_{\epsilon_0} a$ and $|g| < \epsilon_0$ imply (using Lemma~\ref{translation_fact} again) that $$ a + g -g   \sim a - g$$ $$\Rightarrow a \sim a-g,$$ and therefore $a-g \in X$, hence $-g \in X - a$.

\end{proof}

Finally we come to our generalization of Theorem~\ref{sec2main} in the case when $\mathcal{R}$ has burden $2$ instead of dp-rank $2$.  The example in the subsequent section demonstrates the necessity of  the extra hypothesis that $X / \sim$ is finite.

\begin{thm}
\label{dense_finite_twiddle}
If $\mathcal{R}$ is a divisible OAG of burden $\leq 2$ in which an infinite discrete set is definable, then there cannot be a definable set $X \subseteq R$ which is dense and codense in some nonempty interval and such that $X / \sim $ is finite.
\end{thm}

\begin{proof}

Suppose towards a contradiction there is such a set $X$.  Since  $X / \sim $ is finite, there is some $a \in X$ and some nonempty interval $I$ in which $[a]_\sim$ is dense and codense. Let $Z = [a]_\sim \cap I$, and by Lemma~\ref{single_class_reduction}, we have that $| Z / \sim_Z | = 1$.

Since an infinite discrete set is definable in $\mathcal{R}$, by Corollary~\ref{away_dense}, functions on $Z$ are bounded away from zero. Fix any $a \in Z \cap I$, and we may apply Lemma~\ref{single_class_group} to obtain a nonzero convex subgroup $C$ of $\langle R; + \rangle$ such that $H := C \cap (Z -a)$ is a subgroup. By our assumptions on $Z$, $H$ is dense and codense in $C$. Pick any $g \in C \setminus H$.

\begin{claim}
The set $\{\frac{g}{n} \, : \, n \in \N \setminus \{0\} \}$ contains representatives of infinitely many cosets of $H$.
\end{claim}

\begin{proof}
On the one hand, if the image of $g$ in $R / H$ has infinite order, then for any distinct $n, m \in \N \setminus \{0\}$, we have $mg - ng \notin H$, and hence $\frac{g}{n} - \frac{g}{m} = \frac{mg - ng}{nm} \notin H$, so we are done. Otherwise, let $k$ be the least positive natural number such that $k g \in H$, and note that if $\textup{GCD}(k,m) =1$, then $mg \notin H$. Thus whenever $1 \leq i < j$, $$\frac{g}{k^j} - \frac{g}{k^i} = \frac{(1 - k^{j-i} ) g}{k^j},$$ but $\textup{GCD}(1 - k^{j-i} , k) = 1$, so this difference cannot be in $H$.
\end{proof}

By the Claim, the convex set $C$ intersects infinitely many cosets of the group $H$. Now we can easily build a depth-2 inp-pattern within any interval $J \subseteq C$, as follows: for the first row, let

\begin{equation}
\varphi_0(x; a_i, b_i) := a_{0, i} < x < b_{0,i}
\end{equation}
(for $i < \omega$) be any family of formulas which define pairwise disjoint open subintervals of $J$; and for the second row, let

\begin{equation}
\varphi_1(x; c_j) := x - c_j \in X
\end{equation}
with the parameters $\{c_j \, : \, j \in \omega\}$ chosen from $C$ which represent distinct cosets of $H$ (which is possible by the Claim). Since $H$ is a dense and codense subgroup of $C$, every coset $c_j + H$ is also dense and codense in $C$, and hence each pair $\varphi_0(x; a_i, b_i) \wedge \varphi_1(x; c_j)$ is consistent. Thus we have an inp-pattern of depth $2$ in any subinterval $J$ of $C$, contradicting Lemma~\ref{rank_2_intervals} and finishing the proof.
\end{proof}

Thus we have established that for divisible $\mathcal{R}$ we can not simultaneously have an infinite definable discrete set and a definable dense codense set $X$ with $X/\sim$ finite.
Next we consider the case in which $X \subseteq R$ is definable and $X / \sim$ is infinite. First we need a preliminary lemma:

\begin{lem}
\label{dividing_acc_point}
If there is an infinite definable discrete set in $\mathcal{R}$, then there is no definable $X \subseteq R$ such that $X$ divides over some elementary submodel $\mathcal{R}_0$ of $\mathcal{R}$ and $0$ is an accumulation point of $X$.

\end{lem}

\begin{proof}
Otherwise, say there is such a set $X = X_{\overline{a}}$ defined over parameters $\overline{a}$. Pick any interval $I$ around $0$, and we will construct an inp-pattern of depth $2$ in $I$ as follows: first, since $X_{\overline{a}}$ divides, pick a set of $\mathcal{R}_0$-conjugates $\{X_{\overline{a}_i} \, : \, i \in \omega\}$ of $X_{\overline{a}}$ which are $k$-inconsistent for some $k$, and this will be the first row of the inp-pattern. For the second row, we recursively construct a sequence of points

$$b_0 > c_0 > b_1 > c_1 > \ldots > 0$$
as follows: first, let $b_0$ be any positive element of $I$. Since $0$ is an accumulation point of $X$, it is also an accumulation point of each of its conjugates $X_{\overline{a}_i}$, so by saturation we can pick $c_0 > 0$ so that the interval $J_0 := (c_0, b_0)$ intersects each of the sets $X_{\overline{a}_i}$. In general, once we have constructed $b_i$ and $c_i$ for all $i \leq n$, we pick an arbitrary $b_{n+1} \in (0, c_n)$, and then as before the fact that $0$ is an accumulation point of each conjugate of $X_{\overline{a}}$ allows us to pick $c_{n+1} \in (0, b_{n+1})$ so that $J_{n+1} := (c_{n+1}, b_{n+1})$ intersects every $X_{\overline{a}_i}$.

Now the second row in our inp-pattern will consist of the formulas $\varphi_1(x; b_j, c_j)$ asserting that $b_j < x < c_j$, which define pairwise-disjoint intervals which each intersect every set $X_{\overline{a}_i}$, as we wanted.

Since this inp-pattern can be constructed within any interval $I$ around $0$, by Lemma~\ref{rank_2_intervals} we have a contradiction.
\end{proof}

The next proposition is similar to the well-known fact that in an NIP theory, global types which do not divide over a small submodel are invariant over said submodel (see Chapter 5 of \cite{Guide_NIP}, or Proposition 2.1 of \cite{HruPil}). However, we require a version of this which assumes only that the formula we are working with is NIP.
  
\begin{prop}
\label{dividing_invariant}
Suppose $\varphi(\overline{x}; \overline{y})$ is an NIP formula in a theory $T$, $\mathcal{N} \prec \mathcal{M}$ are models of $T$ such that $\mathcal{M}$ is $|N|^+$- saturated, and $p(\overline{x})$ is a complete $\varphi(\overline{x}; \overline{y})$-$M$-type\footnote{That is, a maximal consistent collection of boolean combinations of instances $\varphi(\overline{x}; \overline{b})$ of $\varphi(\overline{x}; \overline{y})$ with $\overline{b}$ from $M$.} which does not divide over $N$. Then $p(\overline{x})$ is $N$-invariant.
\end{prop}

\begin{proof}
The same proof as in \cite{HruPil} goes through. Namely, suppose that $\overline{a}_0, \overline{a}_1 \in M$ and $\tp(\overline{a}_0 / N) = \tp(\overline{a}_1 / N)$, and assume towards a contradiction that $\varphi(\overline{x}; \overline{a}_0) \wedge \neg \varphi(\overline{x}; \overline{a}_1) \in p(\overline{x})$. As $N$ is a model, there is some $\overline{b}$ such that both $\overline{a}_0, \overline{b}$ and $\overline{a}_1, \overline{b}$ can be extended to infinite $N$-indiscernible sequences (see Facts~1.11 and 1.12~(ii) of \cite{CLPZ}), and by saturation we may assume $\overline{b} \in M$. If $\varphi(\overline{x}; \overline{b}) \in p$, then for the $N$-indiscernible sequence $\{\overline{a}'_i  :  i \in \omega\}$ extending $\overline{a}_1, \overline{b}$, we have $\neg \varphi(\overline{x}; \overline{a}'_0) \wedge \varphi(\overline{x}; \overline{a}'_1) \in p(\overline{x})$; and if to the contrary $\neg \varphi(\overline{x}; \overline{b}) \in p$, then let $\{\overline{a}'_i :  i \in \omega \}$ be the $N$-indiscernible sequence extending $\overline{a}_0, \overline{b}$, and we have $ \varphi(\overline{x}; \overline{a}'_0) \wedge \neg \varphi(\overline{x}; \overline{a}'_1) \in p(\overline{x})$. In either case,

\begin{equation}
\neg \left[\varphi(\overline{x}; \overline{a}'_0) \leftrightarrow \varphi(\overline{x}; \overline{a}'_1) \right] \in p(\overline{x}).
\end{equation}

Since $p(\overline{x})$ does not divide over $N$, the partial type $$\{\neg \left[ \varphi(\overline{x}; \overline{a}'_{2i}) \leftrightarrow \neg \varphi(\overline{x}; \overline{a}'_{2i+1}) \right] \, : \, i < \omega \}$$ must be consistent. As any Boolean combination of NIP formulas is NIP, $\neg \left[ \varphi(\overline{x}; \overline{y}_1) \leftrightarrow \neg \varphi(\overline{x}; \overline{y}_2)\right]$ is NIP, but we have just shown that this formula has infinite alternation rank, which is a contradiction.
\end{proof}

\begin{prop}
\label{no_inf_twiddle}
Suppose that there is an infinite discrete set definable in $\mathcal{R}$, $X \subseteq R$ is definable, and the formula $\varphi(x; y)$ expressing that $x \in X - y$ is NIP.  Then $X / \sim$ is finite.
\end{prop}

\begin{proof}
Suppose that $X$ is definable over the $\omega$-saturated model $\mathcal{R}_0$. Let $\mathcal{R}_1$ be an $|R_0|^+$-saturated elementary extension of $\mathcal{R}_0$.  If $X / \sim$ is infinite, then it has unboundedly many equivalence classes (since $\sim$ is a definable equivalence relation) and thus we can find $a, b \in X(R_1)$ such that $\tp(a/ \mathcal{R}_0) = \tp(b / \mathcal{R}_0)$ and $a$ is not in the same $\sim$-class as $b$. This means that for every $\epsilon > 0$, there is some $g \in (-\epsilon, \epsilon)$ such that $$\neg \left[ a + g \in X \leftrightarrow b + g \in X \right],$$ thus $$g \in (X - a) \Delta (X - b),$$ and so $0$ is an accumulation point of the set $Z := (X - a) \Delta (X - b)$. By Lemma~\ref{dividing_acc_point}, $Z$ does not divide over $R_0$. By $\omega$-saturation, $Z$ does not fork over $R_0$ (see, for instance, Proposition~5.14 of \cite{Guide_NIP}). Therefore $Z$ has an extension to a complete $\varphi(x;y)$-type $p(x)$ over $R_1$ which does not fork over $R_0$. But the fact that $a \equiv_{R_0} b$ implies that $Z$ is not $R_0$-invariant, and hence $p(x)$ cannot be $R_0$-invariant either, so by Proposition~\ref{dividing_invariant}, we conclude that $\varphi(x; y)$ cannot be NIP. 

\end{proof}

Putting this all together, we obtain:

\begin{thm}
\label{IP_translations}
Suppose that $\mathcal{R}$ is a divisible OAG of burden $2$ in which an infinite discrete set is definable and also there is a definable $X \subseteq R$ which is dense and codense in some nonempty interval. Then the formula $\varphi(x; y)$ expressing that $x \in X - y$ has the independence property.

\end{thm}

\begin{proof}
By Theorem~\ref{dense_finite_twiddle}, $X / \sim$ is infinite. Hence by Proposition~\ref{no_inf_twiddle}, we conclude that $\varphi(x;y)$ has the independence property.
\end{proof}

\begin{quest}
Are there divisible OAGs of burden $2$ (or dp-rank $2$) in which both a Cantor-like set $C$ and a set $X$ which is dense and codense on an infinite interval are definable?
\end{quest}

\section{A theory with burden $2$ with both an infinite definable discrete set and a dense-codense set}

In this section, we construct an example of a complete theory $T$ with the following properties:  \begin{itemize}
\item $T$ expands the theory divisible ordered Abelian groups;
\item $T$ is definably complete;
\item $T$ has burden 2; and
\item if $\mc{M} \models T$ then there are 
   both an infinite definable discrete subset of $M$ and a definable dense and codense subset of $M$. 

\end{itemize}
   
This establishes that the hypotheses of Theorem \ref{IP_translations} can in fact be satisfied. In order to accomplish this we need to develop some background facts that slightly generalize some of the work in \cite{cp}.

Recall that in \cite{cp} the authors show that given a theory $T$ in a language $\mc{L}$ which eliminates $\exists^{\infty}$ 
and a distinguished unary predicate $S$ from $\mc{L}$ if we expand $\mc{L}$ to $\mc{L}(P)$ by adding a new unary predicate $P$ then 
the $\mc{L}(P)$-theory $$T_{S}= T \cup \{\forall x \left[ P(x) \rightarrow S(x) \right]\}$$ has a model companion $T_G$ with multiple desirable properties.

Here we point out the the results from \cite[section 2]{cp} all hold if we slightly weaken the assumption that $T$ eliminates $\exists^{\infty}$ to the assumption that  $T$ eliminates $\exists^{\infty}$ relative to $S$.

\begin{definition}  For a theory $T$ and a distinguished unary predicate $S$ we say that {\em $T$ eliminates $\exists^{\infty}$ relative to $S$} if for any formula $\varphi(x, \ob{y})$ so that $T \models \forall x \left[\varphi(x, \ob{y}) \rightarrow S(x)\right]$ there is $n \in \NN$ so that if $\mc{M} \models T$ and $\ob{a} \in M^{|\ob{y}|}$ and
 $|\varphi(M, \ob{a}) |> n$ then $\varphi(M, \ob{a})$ is infinite.
 \end{definition}

 We now provide the slight modifications to the result from \cite[Section 2]{cp}  that we will need to construct and analyze our theory.  All the proofs are {\it mutatis mutandis} from those in the original paper of Chatzidakis and Pillay. We include the specific analogous result from \cite{cp} with each statement.  From now on we assume that $T$ eliminates $\exists^{\infty}$ relative to $S$.  We also assume that $T$ eliminates quantifiers.

\begin{definition}\cite[Definitions 2.1]{cp} Let $\mc{M}$ be a saturated model of $T$, $C$ a small subset of $M$, $\ob{a}=(a_1, \dots, a_n)$ a tuple of elements of $M$, and $\varphi(\ob{x})$ a formula defined with parameters in $C$.
\begin{enumerate}

\item We define the algebraic dimension of $\ob{a}$ over $C$, $\text{a-dim}(\ob{a}/C)$, to be the maximal length of a sequence $j(i)$ of positive integers $\leq n$ such that:
\[a_{j(1)}\notin \acl(C), a_{j(i)} \notin \acl(C,a_{j(1)}, \dots, a_{j(i-1)}).\]

\item We set
\[\text{a-dim}(\varphi(\ob{x}))=\sup\{\text{a-dim}(\ob{a}/C) | M \models \varphi(\ob{a})\}.\]
\end{enumerate}
\end{definition}

 \begin{fact}\cite[Lemma 2.2]{cp}  Let $\mc{M} \models T$, $\varphi(x_1, \dots, x_n, \ob{y})$ an $\mc{L}$-formula so that 
 \[T \models \forall x_1, \dots, x_n(\varphi(\ob{x}, \ob{y}) \rightarrow \bigwedge_{1 \leq i \leq n}S(x_i))\] and $d$ an integer.  Then the set $\{\ob{b}\; |\; \text{a-dim}(\varphi(\ob{x},\ob{b}))=d\}$ is definable.
 \end{fact}
 
 \begin{fact}\cite[Lemma 2.3]{cp} Let $\mc{M}$ be a saturated model of $T$ and $\varphi(x_1, \dots, x_n, \ob{y})$ an $\mc{L}$-formula so that 
  \[T \models \forall x_1, \dots, x_n(\varphi(\ob{x}, \ob{y}) \rightarrow \bigwedge_{1 \leq i \leq n}S (x_i)).\]  Then the set 
  $\Sigma(\varphi)$ of tuples $\ob{b}$ so that there is $\ob{a} \in M$ satisfying $\varphi(\ob{x}, \ob{b})$ and $\ob{a} \cap \acl(\ob{b}) =\emptyset$ is definable.
  \end{fact}
  
  \begin{fact}\cite[Theorem 2.4]{cp}\label{cpax}  The theory $T_S$ has a model companion $T_G$, whose axiomatization is obtained by expressing in a first order way the following properties of a model $(M,P)$:
  \begin{enumerate}
  \item $M \models T$.
  \item For every $\mc{L}$-formula $\varphi(x_1, \dots, x_n, \ob{z})$ so that 
   \[T \models \forall x_1, \dots, x_n(\varphi(\ob{x}, \ob{z}) \rightarrow \bigwedge_{1 \leq i \leq n}S(x_i)),\] 
   for every subset $I$ of $\{1, \dots, n\}$,
   \[\forall \ob{z} \left [ \exists\ob{x}\varphi(\ob{x}, \ob{z}) \wedge (\ob{x} \cap \acl_T(\ob{z})= \emptyset )\wedge \bigwedge_{1 \leq i<j \leq n}x_i \not= x_j\right ]\]
   \[\rightarrow \left [\exists\ob{x}\varphi(\ob{x},\ob{z}) \wedge \bigwedge_{i \in I}(x_i \in P) \wedge \bigwedge_{i \notin I}(x_i \notin P)\right]\]
   
   \end{enumerate}
   
   where $\acl_T$ is the algebraic closure operator for $T$.
   \end{fact}
   
 \begin{fact}\cite[Proposition 2.5]{cp}  Let $(\mc{M},P)$ and $(\mc{N},Q)$ be models of $T_G$, and let $A$ be a common subset of $M$ and $N$.  Then 
 \[(\mc{M},P)\equiv_A(\mc{N},Q) \Leftrightarrow (\acl_T(A),P \cap \acl_T(A)) \simeq_A (\acl_T(A), Q \cap \acl_T(A)).\]
 \end{fact}
 
 \begin{fact}\cite[Corallaries 2.6]{cp}\label{completions} \begin{enumerate} \item  The completions of $T_G$ are obtained by describing $P \cap \acl_T(\emptyset)$.
 \item If $\ob{a}, \ob{b}$ are tuples from $\mc{M}\models T_G$ and $A \subseteq M$, then $\tp(\ob{a}/A)=\tp(\ob{b}/A)$ if and only if there is an $A$-isomorphism of $\mc{L}(P)$ structures from $\acl_T(A, \ob{a})$ to $\acl_T(A, \ob{b})$ which carries $\ob{a}$ to $\ob{b}$.
 
 \item Let $a \in \mc{M} \models T_G$, $A \subseteq M$.  Then $a$ is algebraic over $A$ if and only if $a \in \acl_T(A)$.  Thus algebraic closures in the sense of $T$ and $T_G$ coincide.
 
 \item Modulo $T_G$, every formula $\varphi(\ob{x})$ is equivalent to a disjunction of formulas of the form $\exists \ob{y} \psi(\ob{x}, \ob{y})$, where $\psi$ is quantifier-free, and for every $(\ob{a}, \ob{b})$ satisfying $\psi$, $\ob{b} \in \acl_T(\ob{a})$.
 
 \end{enumerate}
 \end{fact}
 
 Given this general background information we now proceed to analyze the specific theory $T$ with which we plan to work.
 
Let $T$ be the theory of the structure $\langle \R; +, <, \ZZ\rangle$ formulated in the language
$\mc{L}=\{+, <, 0, 1, \floor{\;}, S, \lambda\}_{\lambda \in \Q}$ where $S$ is interpreted as the interval $(0,1)$, the $\lambda$ are simply multiplication by $\lambda \in \Q$ and 
 $\floor{x}=\max\{y \in \ZZ : y\leq x\}$.  Note we do not include a separate predicate for $\ZZ$ as the set of integers is defined by the quantifier free formula $\floor{x}=x$.  Also notice that as $T$ is linearly ordered $\acl=\dcl$ in models of $T$.  By results from the appendix of \cite{ivp} the $\mc{L}$-theory $T$ eliminates quantifiers and is universally axiomatizable.  We record some additional facts about $T$ which we will use repeatedly without further mention.
 
 \begin{fact}\label{zfact}
 \begin{enumerate}
 \item If $\mc{M} \models T$ and $f:M^n \to M$ is definable then $f$ is given piecewise by terms; more precisely, if $f$ is definable with parameters $\ob{c}$, then there are $\ob{c}$-definable sets $X_1, \dots, X_k$ partitioning $M^n$ and terms $t_1(\ob{x}, \ob{y}), \dots, t_k(\ob{x}, \ob{y})$ so that $f(\ob{x})=t_i(\ob{x}, \ob{c})$ on $X_i$.
   \item $T$ has dp-rank 2.
 \item If $\mc{M} \models T$ and $X \subseteq M$ is definable then $X$ is either discrete or has interior.
 \item If $\mc{M} \models T$ and $X \subseteq (0,1)$ is definable then $X$ is a finite union of points and intervals, so $T$ eliminates $\exists^{\infty}$ relative to $S$ and the induced structure on $(0,1)$ is o-minimal.
 \item If $\mc{M} \models T$, $X \subseteq M$ is definable and discrete and $f: M \to M$ is definable then $f[X]$ is discrete.
 \item If $\mc{M} \models T$, $U \subseteq M$ is an open interval and $f: U \to M$ is definable then there is an open subinterval $V \subseteq U$ so that $f$ is equal to a linear function of the form 
 $x \mapsto \lambda x +a$ on $V$.
 \end{enumerate}
\end{fact}

\begin{proof} For (1), it is well known that this follows from the fact that $T$ is universally axiomatizable and model complete.  For completeness we outline a proof.  Fix $\mc{M}$ and $f$.  Assume that $f$ is definable with parameters $\ob{c}$.  Let $\ob{a} \in M$.  Consider the closure of $\ob{a}\ob{c}$ under all terms, $\langle \ob{a}\ob{c}\rangle$.  As $T$ is universal $\langle \ob{a}\ob{c}\rangle \models T$.  As $T$ is model complete $\langle \ob{a}\ob{c}\rangle \preceq \mc{M}$.  In particular $f(\ob{a}) \in \langle \ob{a}\ob{c}\rangle$ so that $f(\ob{a})=t(\ob{a}, \ob{c})$ for some term $t$.  The result then follows by compactness.

  See the note after the proof of Theorem 3.1 in \cite{DG} for (2).  As $T$ is clearly definably complete (3) follows from Corollary \ref{discrete_interior_dichotomy}.  (4) is  a special case of \cite[Lemma 3.3(1)]{DG}.  (5) follows from \cite[Corollary 2.17]{DG}.  (6) follows from \cite[Lemma 3.2(2)]{DG} and  (1).
\end{proof}

As $T$ eliminates $\exists^{\infty}$ relative to $S$ we can form the theory $T_G$.  For convenience we fix a completion $T^*_G$ of $T_G$ by specifying that if $\mc{M} \models T_G$ then $P(M) \cap \dcl{(\emptyset)}=\emptyset$ (see Fact \ref{completions}(1)).

\begin{lem} $T_G^*$ has quantifier elimination and is definably complete.
\end{lem}

\begin{proof}  That $T^*_G$ has quantifier elimination follows from Fact \ref{completions}~(4) together with the facts that $\acl_T=\dcl_T$, definable functions are given piecewise by terms (by Fact \ref{zfact}~(1)), and a use of compactness.

To establish that $T_G^*$ is definably complete  we work in the structure $\mc{R}=\langle \R; +, <, 0, 1, \floor{\;}, S, \lambda\rangle_{\lambda \in \Q}$.  It suffices to show that there is $G \subseteq \R \setminus \dcl(\emptyset)$ so that \[\langle \R; +, <, 0, 1, \floor{\;}, S, G, \lambda\rangle_{\lambda \in \Q}\] is a model of $T_G^*$.

We construct $G \subseteq \R$ and its complement $\ob{G} = \R \setminus G$ simultaneously by induction.  Let $\bm{c}=\left|\R\right|$. 

 By an instance of axiom (2) in Fact \ref{cpax} we mean a triple $C=(\varphi(\ob{x}, \ob{y}), \ob{c}, I)$ so that:
 \begin{enumerate}
 \item  $\varphi(x_1, \dots, x_n, \ob{y})$ is an $\mc{L}$-formula;
  \item $\ob{c} \in \R^{|\ob{y}|}$ and in some extension $\mc{M}$ of $\mc{R}$ there is $\ob{a} \in M^{|\ob{x}|}$ satisfying $\varphi(\ob{x}, \ob{c})$ with $\ob{a} \cap \acl{(\ob{c})}=\emptyset$ and all coordinates of $\ob{a}$ distinct; and
 \item $I \subseteq \{1, \dots, n\}$.
 \end{enumerate}
 We say for a set $G' \subseteq \R$  a tuple  $\ob{a} \in \R^{|\ob{x}|}$ satisfies $C$ relative to $G'$  if $\mc{R} \models \varphi(\ob{a}, \ob{c})$ and $a_i \in G'$ if and only if $i \in I$.

 List all instances of axiom (2) as $C_{\alpha}=(\varphi_{\alpha}(\ob{x},\ob{y}), \ob{c}_{\alpha}, I_{\alpha})$ for $\alpha \in \bm{c}$.
 
  We construct $G$ as $\bigcup_{\alpha \in \bm{c}}G_{\alpha}$ and $\ob{G}=\bigcup_{a \in \bm{c}}\ob{G}_\alpha$ so that:
\begin{enumerate}
\item $|G_{\alpha}|<\bm{c}$ and $|\ob{G}_{\alpha}|<\bm{c}$ for all $\alpha<\bm{c}$;
\item $G_{\alpha} \subseteq G_{\beta}$ and $\ob{G}_{\alpha}  \subseteq \ob{G}_{\beta}$ for $\alpha<\beta$;
\item $G_0=\emptyset$ and $\ob{G}_0=\dcl(\emptyset)$;
\item $G_{\alpha} \cap \ob{G}_{\alpha}=\emptyset$ for all $\alpha<\bm{c}$; and
\item there is $\ob{c}_{\beta} \in G_{\alpha} \cup \ob{G}_{\alpha}$ satisfying condition   $C_{\beta}$ relative to $G_{\alpha}$ for all $\beta < \alpha$.
\end{enumerate}

If we can construct such a $G$ then it is immediate that \[\langle \R; +, <, 0, 1, \floor{\;}, S, G, \lambda\rangle_{\lambda \in \Q} \models T_G^*.\]

Constructing $G_0$ and $\ob{G}_0$ is immediate.  Also if $\alpha$ is a limit ordinal and we have constructed $G_{\beta}$ and $\ob{G}_{\beta}$ for all $\beta<\alpha$ then we simply let $G_{\alpha}=\bigcup_{\beta<\alpha}G_{\alpha}$ and $\ob{G}_{\alpha}=\bigcup_{\beta<\alpha}\ob{G}_{\beta}$.   

Thus suppose we have a successor ordinal $\alpha+1$ and we have constructed $G_{\beta}$ and $\ob{G}_{\beta}$ for all $\beta \leq\alpha$.  We need to extend the construction to satisfy $C_{\alpha}=(\varphi_{\alpha}(\ob{x}, \ob{y}), \ob{c}_{\alpha}, I_{\alpha})$.  Let $X_{\alpha} \subseteq (0,1)^n$ be the subset of $\R^n$ defined by $\varphi_{\alpha}(\ob{x},\ob{c}_{\alpha})$.  By Fact \ref{zfact}~(2) the structure induced by $\mc{R}$ on $(0,1)$ is o-minimal.  Hence without loss of generality we may assume that $X_{\alpha}$ is an o-minimal cell.
After potentially permuting the variables of the defining formula for $X_{\alpha}$ we may without loss of generality assume that 
\[X_{\alpha}=\{(z_1, \dots, z_l, y_1, \dots, y_m) : \ob{z} \in U \text{ and } y_i=g_i(\ob{z}) \text{ for all } 1 \leq i \leq m\}\] 
where $U \subseteq \R^l$ is a definable  open o-minimal cell and the $g_i$ are definable continuous functions on $U$ which are never equal on $U$.  After refining $U$ if necessary we may also assume that if $a_1, \dots, a_{l-1} \in \pi(U)$ (where $\pi$ is projection on the first $l-1$ coordinates) then $g_i: U_{(a_1, \dots, a_{l-1})} \to \R$ is monotone  for all $1 \leq i \leq m$ (where $U_{(a_1, \dots, a_{l-1})}$ is the fiber in $U$ over $(a_1, \dots, a_{l-1})$).  Furthermore we may also assume that for any $1 \leq i \leq m$ if 
for some $(a_1, \dots, a_{l-1}) \in \pi(U)$ the function $g_i:U_{(a_1, \dots, a_{l-1})} \to \R$ is increasing then the same is true for all $(b_1, \dots b_{l-1}) \in \pi(U)$ and the same holds for ``decreasing'' or ``constant'' in place of ``increasing''.

To build $G_{\alpha+1}$ and $\ob{G}_{\alpha+1}$ it suffices to find $a_1, \dots, a_l, b_1, \dots, b_m \in X_{\alpha}$ so that 
$a_i \notin G_{\beta} \cup \ob{G}_{\beta}$ and $b_j \notin G_{\beta}\cup \ob{G}_{\beta}$ for all $1 \leq i \leq l$, $1 \leq j \leq m$, and $\beta \leq \alpha$.  We show this by induction on $l$, the dimension of $U$.  If $l=0$ then $U$ is a point, but then there can not be $\ob{a}$ in an extension of $\mc{R}$ with $\ob{a} \in X_{\alpha}$ and $\ob{a} \cap \acl(\ob{c}) =\emptyset$.  Hence there can be no such instance of axiom (2).  Thus we may assume that $l>0$ we have our result for all values less that $l$.   Suppose that for some function $g_j$ with $1 \leq j \leq m$ it is the case that $g_j: U_{(a_1, \dots, a_{l-1})} \to \R$ is constant (for simplicity suppose that this holds for the functions $g_1, \dots, g_r$ with $0 \leq r \leq m$).  Thus for $1 \leq i \leq r$ the function $g_i$ may be thought of as a function from $\pi(U)$ to $\R$.  By induction we may find $a_1, \dots a_{l-1} \in \pi(U)$ so that 
$a_i \notin G_{\beta} \cup \ob{G}_{\beta}$ for any $1 \leq i \leq l-1$ and $\beta \leq \alpha$ and also so that $g_i(a_1, \dots, a_{l-1}) \notin G_{\beta} \cup \ob{G}_{\beta}$ for all $1 \leq i \leq r$ and all $\beta \leq \alpha$.  If there are no constant functions $g_i$ then we can also easily pick $a_1, \dots, a_{l-1} \in \pi(U)$ so that no $a_i \in G_{\beta} \cup \ob{G}_{\beta}$ for any $\beta\leq\alpha$.  Each function $g_i$ with $r<i$ is monotone increasing or decreasing on $U_{(a_1, \dots, a_{l-1})}$, in particular there are less than $\bm{c}$ elements $a_l$ of $U_{a_1, \dots, a_{l-1}}$ so that $g_i(a_1, \dots, a_l) \in G_{\beta} \cup\ob{G}_{\beta}$ for any $\beta\leq\alpha$.  Also that are less that $\bm{c}$ element $a_l$ of $U_{(a_1, \dots, a_{l-1})}$ so that $a_l \in G_{\beta} \cup\ob{G}_{\beta}$ for some $\beta\leq\alpha$.  Thus simply due to cardinality considerations there must be $a_l \in U_{(a_1, \dots, a_{l-1})}$ so that $a_l \notin G_{\beta} \cup \ob{G}_{\beta}$ for any $\beta\leq \alpha$ and $g_i(a_1 \dots, a_l) \notin G_{\beta} \cup \ob{G}_{\beta}$ for all $1 \leq i \leq m$ and $\beta\leq \alpha$.  Let $\ob{c}$ be the $n$-tuple $(a_1, \dots, a_l, g_1(\ob{a}), \dots, g_m(\ob{a}))$.  Set $G_{\alpha+1}=G_{\alpha} \cup \{c_i : i \in I_{\alpha}\}$ and $\ob{G}_{\alpha+1}=\ob{G}_{\alpha} \cup \{c_i : i \in \{1, \dots, n\} \setminus I_{\alpha}\}$.

\end{proof}

Axiom scheme (2)  for $T_G$  from Fact \ref{cpax} implies that if $\mc{M} \models T_G^*$  then $P(M)$ must be dense and codense in $(0,1)$.  Thus to establish that $T_G^*$ is our desired example we must show that $T_G^*$ is of burden $2$.  Ideally we would simply like to reference \cite[Theorem 7.3]{chernikov} but as $T$ does not satisfy exchange for algebraic closure it is not clear that this result is applicable, so we provide an \emph{ad hoc} proof. We need some preliminary lemmas and observations.  The following is an immediate consquence of the axiomatization for $T_G$ given in Fact \ref{cpax}.

\begin{fact}\label{genfact} Let $\mc{M} \models T_G$ and let $U \subseteq M$ be an open interval. Furthermore suppose that $f_i: U \to (0,1)$ for $1 \leq i \leq n$ and $g_j: U \to (0,1)$  for $1 \leq j \leq m$ are definable functions which are continuous, non-constant, and monotone.  If there is $V \subseteq U$ an open interval so that $f_i(x) \not= g_j(x)$ for all $x \in V$ and all $i,j$ then there is $x \in V$ so that 
$f_i(x) \in P(M)$ for all $i$ and $g_j(x) \notin P(M)$ for all $j$.
\end{fact}

\begin{lem}\label{loc_const}  Let $\mc{M} \models T_G^*$ and $f: M \to (0,1)$ be definable.  Then there are only finitely many $a \in (0,1)$ so that $f^{-1}(a)$ has nonempty interior.
\end{lem}

\begin{proof} Let 
\[X=\{b \in M : f \text{ is constant in a neighborhood of } b\}.\]  Either $X$ is empty in which case we are done or it is a nonempty open set and hence a union of open intervals by definable completeness.  Also by definable completeness if $I$ is any one of these intervals $f$ is constant on $I$.  Let $X_0$ be the set of midpoints of these intervals.  $X_0$ is definable and discrete. Thus $f[X_0]$ is discrete and as $f[X_0] \subseteq (0,1)$ it must be finite.  This establishes the lemma.  

\end{proof}

\begin{lem}\label{good_form} Let $\mc{M} \models T_G^*$ and let $X \subseteq M$ be definable.  $X$ is a finite union of sets definable by formulas of the form:
\[\psi(x) \wedge \bigwedge_{i \in I}P(t_i(x)) \wedge \bigwedge_{j \in J}\neg P(s_j(x))\]
where $\psi(x)$ is an $\mc{L}$-formula possibly with parameters, $I$ and $J$ are finite sets, and the $t_i$'s and $s_j$'s are terms possibly with parameters so that:
\begin{enumerate}
\item $\psi(M)$ is either discrete or open;
\item if $\mc{M} \models \psi(a)$ then $t_{i_1}(a) \not= t_{i_2}(a)$ for $i_1 \not= i_2$, $s_{j_1}(a) \not= s_{j_2}(a)$ for $j_2 \not= j_2$ and $t_i(a) \not= s_j(a)$ for all $i,j$;
\item if $\mc{M} \models \psi(a)$ then $t_i(a) \in (0,1)$ and $s_j(a) \in (0,1)$ for all $i,j$; and
\item if $I$ is an open interval with $I \subseteq \psi(M)$ then no $t_i$ and no $s_j$ is constant on $I$.
\end{enumerate}
\end{lem}

\begin{proof} By quantifier elimination for $T_G$ we reduce to the case where $X$ is defined by a formula of the form:
\[\psi_0(x) \wedge \bigwedge_{i \in I_0}P(t_i^0(x)) \wedge \bigwedge_{j \in J_0}\neg P(s^0_j(x))\] where $\psi_0$ is an $\mc{L}$-formula, $I_0$ and $J_0$ are finite sets, and the $t^0_i$'s and $s^0_j$'s are terms.  The proof is a straightforward induction on $|I_0+J_0|$ and proceeds by repeatedly partitioning the set defined by $\psi_0(x)$ into smaller $\mc{L}$-definable subsets where (1) follows as any definable set in $\mc{M}$ either is discrete or has interior, (2) is immediate, (3) is immediate as $P(M) \subseteq (0,1)$, and (4) follows by Lemma \ref{loc_const} and the fact that $T$ eliminates $\exists^{\infty}$ relative to $(0,1)$.

\end{proof}

\begin{lem}\label{disc_def}  If $\mc{M} \models T_G^*$ and $X \subseteq M$ is definable and discrete then $X$ is definable in 
$\mc{M} \restriction \mc{L}$.
\end{lem}

\begin{proof} Without loss of generality we may assume that $X$ is defined by a formula $\varphi(x)$ of the form given in Lemma \ref{good_form}, say	
\[\varphi(x) = \psi(x) \wedge \bigwedge_{i \in I}P(t_i(x)) \wedge \bigwedge_{j \in J}\neg P(s_j(x)).\]

 If $\psi(M)$ has interior we can find an open interval $U \subseteq \psi(M)$ so that each $t_i$ and $s_j$ is equal to a linear term of the form $\lambda x + a$ on $U$.  By assumption none of the $s_i$ and $t_i$ are equal on $U$ and none of them are constant.  Thus by Fact \ref{genfact}  it must be the case that $\varphi(M)$ is dense in $U$ and thus this case is impossible.
 
If $\psi(M)$ is discrete the image of $\psi(M)$ under any of the $t_i$ or $s_j$ must be discrete.   As these images are subsets of $(0,1)$ they must be finite.  It follows that $X$ is definable in $\mc{M} \restriction \mc{L}$.

\end{proof}

Although we only need the following lemma in the specific case of models of $T$ we provide a statement and proof in much greater generality. 

\begin{lem}\label{disc_rank}  Suppose that $\mc{N}$ is an expansion of a densely ordered group of finite burden $n$.  If $X \subseteq N$ is definable, infinite, and discrete, then the burden of $X$ is less than or equal to $n-1$.
\end{lem}

\begin{proof}  Without loss of generality we assume that $\mc{N}$ is sufficiently saturated.  Suppose the result fails.  Thus $X$ has burden $n$.    Let the formulas $\varphi_i(x, \ob{y})$ for $1 \leq i \leq n$ and parameters $\ob{a}_{i,j}$ for $1 \leq i \leq n$ and $j \in \omega$ witness that $X$ has burden $n$.  Notice that we only need countably many elements from $X$ to witness that the burden is $n$, thus by compactness we can find $\delta>0$ so that 
\[X'=\{x \in X : (x-\delta, x+\delta) \cap X =\{x\}\}\] also has burden $n$ witnessed by the same formulas and parameters.

We can find $0<\varepsilon_j^0<\varepsilon_j^1<\delta$ for $j \in \omega$ so that $\varepsilon^1_j<\varepsilon^0_{j+1}$ for all $j \in \omega$.  Let $\varphi_{n+1}(x,y_1y_2)$ be the formula
\[ \exists z \in X'(z<x \wedge \forall w(w \in X' \rightarrow w \leq z \vee w>x) \wedge y_1<x-z<y_2)\]
and let $\varphi'_i(x,\ob{y})$ for $1 \leq i \leq n$ be the formula
\[\exists z (z \in X' \wedge \forall w(w \in X' \rightarrow w \leq z \vee w>x) \wedge \varphi_i(x, \ob{y})).\]
It now follows that the formulas $\varphi'_1(x, \ob{y}) \dots \varphi'_n(x, \ob{y})$ and $\varphi_{n+1}(x,y_1y_2)$ with respective sequences of parameters $\ob{a}_{1,j}, \dots \ob{a}_{n,j}$ and $\epsilon^0_j,\epsilon^1_j$ witness that the theory of $\mc{N}$ has burden at least $n+1$, a contradiction.
\end{proof}

\begin{prop}  $T_G^*$ has burden $2$.
\end{prop}

\begin{proof} Fix $\mc{M}$ a sufficiently saturated model of $T$ and  suppose the result fails.  Let $\varphi_k(x, \ob{y})$ for $k \in \{1, 2, 3\}$ together with mutually indiscernible parameters $\ob{a}_{k,l}$ with $k \in \{1, 2, 3\}$ and $l \in \R$ be an inp-pattern with three rows.

If any of the $\varphi_k(x, \ob{a}_{k,0})$, say $k=1$, defines a discrete set $X$ then by Lemma \ref{disc_def} $X$ is definable in $\mc{M} \restriction \mc{L}$.  As $\mc{M} \restriction \mc{L}$ has dp-rank $2$,  $X$ must have dp-rank $1$ in $\mc{M} \restriction \mc{L}$ by Lemma \ref{disc_rank}.   Now consider $\varphi_2(M, \ob{a}_{2,0}) \cap X$.  This set is discrete and hence also definable in $\mc{M} \restriction \mc{L}$, say by the $\mc{L}$-formula $\theta_2(x, \ob{b}_{2,0})$.  As  $\{\ob{a}_{2,i} : i \in \R\}$ is indiscernible over $\ob{a}_{1,0}$ we find $\{\ob{b}_{2,i} : i \in \R\}$ so that $\theta_2(x, \ob{b}_{2,i})$ defines $X \cap \varphi_2(x, \ob{a}_{2,i})$ for $i \in \R$.  Similarly we find an $\mc{L}$-formula $\theta_3(x, \ob{y})$ and parameters $\{\ob{b}_{3,i} : i \in \R\}$ so that $\theta_3(x, \ob{b}_{3,i})$ defines $X \cap \varphi_3(x, \ob{a}_{3,i})$ for all $i \in \R$.  Hence the pair of formulas $\theta_2(x, \ob{y})$ and $\theta_3(x, \ob{y})$ with respective sequences of parameters $\{\ob{b}_{2,i} : i \in \R\}$ and $\{\ob{b}_{3,i} : i \in \R\}$ witnesses that $X$ has burden $2$.  But then $X$ has dp-rank $2$, a contradiction.

Consider $\varphi_1(x, \ob{y})$.  By  \cite[Lemma 7.1]{chernikov} we may assume  that this formula is of the form given by Lemma \ref{good_form}, say
\[\varphi_1(x, \ob{y}) = \psi_1(x, \ob{y}) \wedge \bigwedge_{i \in I}P(t_i(x, \ob{y})) \wedge \bigwedge_{j \in J}\neg P(s_j(x, \ob{y}))\] with $\psi_1(M, \ob{a}_{1,0})$ open.

Suppose that $\{\psi_1(x, \ob{a}_{1,l}) : l \in \R\}$ is consistent. 
 Then by compactness we may find an open interval $V=(c,d)$ so that $V$ is a subset of $\psi_1(M, \ob{a}_{1,l})$ for all $l \in \R$.  By shrinking $V$ further we may also assume  by Fact \ref{zfact}(6) that all the terms $t_i(x, \ob{a}_{1,l})$ and $s_j(x, \ob{a}_{1,l})$ are given by linear functions of the form $\lambda x +a$ on $V$.
   As $\mc{M} \restriction \mc{L}$ is of dp-rank 2, by \cite
   [Theorem 4.16]{Guide_NIP} we may find an open interval $W \subset \R$
  so that $\{\ob{a}_{1,j} : j \in W\}$ is indiscernible over $cd$ as a sequence in $\mc{M} \restriction \mc{L}$.  For notational convenience assume that $W=\R$. 
  
 Fix $\kappa \in \omega$ such that $\{\varphi_1(x, \ob{a}_{1,l}) : l \in \R\}$ is $\kappa$-inconsistent.  By the assumptions on the formula $\varphi_1(x, \ob{y})$ from Lemma \ref{good_form}  all the terms are continuous, monotone, and non-constant on $V$ and we have guaranteed that they are also all linear on $V$.  But then by Fact \ref{genfact} we can only have $\kappa$-inconsistency of $\{\varphi_1(x, \ob{a}_{1,l}) : l \in \R\}$  if for some $i \in I$, some $j \in J$, and some $l_1, l_2 \in \R$,
  we have that $t_i(x, \ob{a}_{1, l_1})=s_j(x, \ob{a}_{1, l_2})$ densely often in some subinterval of $V$.   As these functions are linear over $\Q$, it follows that  $t_i(x, \ob{a}_{1, l_1})=s_j(x, \ob{a}_{1, l_2})$ on all of $V$.  Then by indiscernibility of $\ob{a}_{1,l}$ over $cd$ (as an $\mc{M} \restriction \mc{L}$-sequence) it follows that $t_i(x, \ob{a}_{1,0})=s_j(x, \ob{a}_{1,0})$ on $V$.  But this violates the properties of $\varphi_1$ guaranteed by Lemma \ref{good_form}.
  
  Thus it must be the case that $\{\psi_1(x, \ob{a}_{1,l}) : l \in \R\}$ is inconsistent.  But the same holds for 
  $\{\psi_k(x, \ob{a}_{k,l}) : l \in \R\}$ for $k \in \{2,3\}$.  Thus $\psi_k(x, \ob{y})$ for $k \in \{1,2,3\}$ and $\ob{a}_{k,l}$ for $k \in \{1,2,3\}$ and $l \in \R$ is an inp-pattern with three rows in $\mc{M}\restriction \mc{L}$ which is impossible as $\mc{M}\restriction \mc{L}$ has dp-rank $2$.

\end{proof}

This example demonstrates that the results in Section 2 of this paper are in some sense sharp.  In particular notice that as indicated in Theorem \ref{IP_translations} the formula $\tau(x,y):=x \in P-y$ has the independence property by the axioms for $T_G$.

\bibliography{modelth}

\end{document}